\documentclass{article}

\usepackage{amsmath, amsthm, amssymb, amsfonts}
\usepackage{dsfont}
\usepackage{float}
\usepackage{algorithm}
\usepackage{algpseudocode}
\usepackage{mathrsfs}
\usepackage{graphicx}
\usepackage{graphics}
\usepackage[T1]{fontenc}
\usepackage{pstricks}
\usepackage{pst-plot}
\usepackage{pst-all}
\usepackage{pict2e}
\usepackage{parskip}
\usepackage{appendix}
\usepackage{hyperref}
\usepackage{verbatim}

\newtheorem{Theorem}{Theorem}

\newtheorem{Lemma}[Theorem]{Lemma}
\newtheorem{Proposition}[Theorem]{Proposition}
\newtheorem{Remark}[Theorem]{Remark}

\def\1e{\mathds{1}}

\def\Ce{\mathds{C}}

\def\Ne{\mathds{N}}

\def\Re{\mathds{R}}

\def\Te{\mathds{T}}

\def\Ze{\mathds{Z}}

\def\vg0{\mathbf{0}}

\def\ran{\mathop{\mathrm{ran}}}
\def\ker{\mathop{\mathrm{ker}}}

\def\ud{\mathrm{d}}

\def\eqd{: =}

\def\segcc#1#2{[#1, #2]}
\def\segco#1#2{[#1, #2)}
\def\segoc#1#2{(#1, #2]}

\def\vphan{\vphantom{\bigl|}}

\def\scal#1#2{\langle #1,#2 \rangle}

\def\set#1#2{\{\mskip 1mu #1 \mskip 1mu
    | \mskip 1mu #2 \mskip 1mu \}
    }
\def\setc#1#2{
    \left\{
    \mskip 2mu #1 \mskip 2mu
    \left| \vphan\vphantom{#1#2} \right.
    \mskip 2mu #2 \mskip 2mu
    \right\}
    }

\def\mod#1{|\mskip 1mu #1 \mskip 1mu|}

\def\norm#1{\| \mskip 1mu #1 \mskip 1mu \|}
\def\normc#1{
    \left\|
    \mskip 2mu #1 \vphan \mskip 2mu
    \right\|
    }

\def\ran#1{\mathop{\mathrm{ran}}#1}

\def\supp#1{\mathop{\mathrm{supp}}#1}

\title{A variational mollification approach to circular deconvolution}
\author{Mirza Karamehmedovi\'c, Pierre Mar\'echal, and Mykhailo Petrov}
\date{}

\begin{document}

\maketitle

\begin{abstract}
We propose a variational mollification approach to circular deconvolution based on the reconstruction of a mollified target object rather than the exact solution itself. The method is formulated as a convex variational problem whose solution admits an explicit Fourier representation. We establish the consistency of the proposed reconstruction as the target resolution increases and derive convergence rates for deterministic data perturbations.
For ordinary smooth kernels, the method achieves the classical order-optimal algebraic convergence rates under Besov-Nikolskii smoothness assumptions, while for supersmooth kernels it attains the corresponding order-optimal logarithmic rates. Numerical experiments on synthetic and wind-direction data illustrate the effectiveness of the proposed approach and confirm the theoretical predictions.
\end{abstract}

\section{Introduction}

Inverse problems arise in a wide variety of scientific disciplines,
including signal and image processing, optics, medical imaging,
geophysics, astronomy and inverse scattering.
Their common objective is to reconstruct an unknown object from indirect
measurements related through an operator equation.
Whenever the inverse operator fails to be continuous,
small perturbations of the data may produce arbitrarily large errors in
the reconstruction, making regularization an essential ingredient of any
reliable inversion procedure.
The mathematical foundations of regularization theory are now classical
and have been extensively developed since the pioneering works of
Tikhonov, Morozov and many others
\cite{tikhonov1977,morozov1984,engl1996,kirsch2011,scherzer2009}.

Classical regularization theory is primarily concerned with recovering
the exact unknown while controlling the instability inherent to the
inverse problem.
In many practical situations, however, the available measurements do not
contain enough information to justify such an objective.
Because only a limited range of frequencies can generally be reconstructed
reliably, it is often more natural to seek a stable approximation of the
unknown whose resolution is compatible with the information carried by
the data.

This line of thought can already be traced back to the pioneering work of
Lannes and collaborators \cite{lannes1987stabilized} on stabilized
Fourier synthesis.
Their starting point is the observation that spectral interpolation
remains stable whereas spectral extrapolation is intrinsically
ill-posed.
Rather than attempting an unstable reconstruction of the complete
spectrum, they advocate limiting the resolution of the reconstructed
object to what is supported by the available frequency information.
Although the notion of a target object is not formulated explicitly in
modern mathematical terms, it is already present through the concept of a
physically meaningful object representation determined by the chosen
resolution limit.

Within inverse problems, this philosophy was subsequently placed on a
rigorous mathematical footing through the theory of approximate inverses
developed by Louis and Maass
\cite{louis1989,louis1990}.
Instead of attempting to reconstruct the exact solution itself, they
prescribe a family of mollified target objects and construct
reconstruction operators converging towards these objects as the
resolution parameter tends to zero.
This approach provides a natural compromise between stability,
resolution and interpretability.

Motivated by this philosophy, the concept of variational mollification
was introduced in \cite{alibaud2009variational} for the inversion of
truncated Fourier operators.
Rather than regarding the resolution parameter as a purely numerical
regularization parameter, it is interpreted as defining a family of
reconstruction targets connected through a genuine asymptotic theory.
More precisely, it was shown that the variationally reconstructed target
converges towards the exact object as the resolution parameter tends to
zero, thereby establishing the consistency of the approach.

More recently, this framework was extended to probability density
deconvolution in \cite{hohage2024mollifier}, where deterministic
convergence rates were established under Besov-Nikolskii source
conditions.
This demonstrated that variational mollification not only provides an
interpretable reconstruction target but also achieves quantitative
regularization guarantees comparable with those of classical
regularization methods.

The present work continues this line of research in the setting of
circular deconvolution.
Circular convolution naturally arises in numerous applications involving
periodic or angular data, including directional statistics,
orientation estimation, phase reconstruction and harmonic analysis on the
circle.
It also constitutes a natural model problem for Fourier synthesis and
spectral reconstruction.
Several approaches to circular deconvolution have been proposed in the
statistical literature, including adaptive spectral cut-off and model
selection methods
\cite{van1991regularized,johannes2013adaptive,Johannes2021}.
Related questions also arise in Fourier synthesis and super-resolution,
where one seeks to recover functions from incomplete frequency
measurements
\cite{isaev2020holder,isaev2022reconstruction,isaev2023super}.

The objective of the present paper is twofold.
First, we derive a variational reconstruction formula built around a
prescribed mollified target object, thereby retaining a direct
resolution-based interpretation of the reconstruction.
Second, we provide a complete deterministic convergence analysis of this
reconstruction.
Under ordinary smooth assumptions on the convolution kernel, we prove
order-optimal algebraic convergence rates in Besov-Nikolskii spaces.
Under supersmooth assumptions, we recover the classical logarithmic
convergence rates characteristic of exponentially ill-posed inverse
problems.
These results show that the variational mollification approach achieves
the same asymptotic performance as classical regularization methods while
preserving a direct interpretation of the reconstructed object.

Beyond circular deconvolution itself, the underlying philosophy advocated
in this paper is that the choice of the target object should precede the
construction of the inversion algorithm.
From this perspective, convergence analysis becomes the second step of
the theory: once an interpretable reconstruction target has been defined,
one may investigate consistency, stability and convergence rates.
We believe that this viewpoint provides a natural bridge between the
physical notion of finite resolution and the mathematical theory of
regularization.

The paper develops this program progressively.
Section~\ref{sec:cuc} introduces the prescribed mollified target object
and derives the associated variational reconstruction formula.
Section~\ref{sec:duc} investigates the asymptotic behavior of this
reconstruction, establishing consistency and order-optimal convergence
rates in both the ordinary smooth and supersmooth settings.
Section~\ref{sec:num} concludes with numerical experiments illustrating the theoretical findings and the practical performance of the proposed approach.

\section{Convolution on the unit circle}
\label{sec:cuc}

To fix the notation, and for convenience,
we start with some basic facts on function
spaces defined on circles. Let $\Te=\Re/2\pi\Ze$
denote the unit circle. Throughout the paper,
integrals on $\Te$ are taken with respect
to the normalized Haar measure
\[
\ud\sigma(t)=\frac{\ud t}{2\pi}.
\]
We denote by $L^p(\Te)$ the
space of Borel measurable
functions on~$\Te$ such that
\[
\int_\Te\mod{f(t)}^p\,\ud\sigma(t)<\infty.
\]
For $f\in L^1(\mathbb T)$, we define
\[
\widehat f(k)=\int_{\mathbb T}
f(t)e^{-ikt}\,\ud\sigma(t),
\quad k\in\Ze.
\]
In particular,
\[
\widehat f(0)=
\int_{\mathbb T} f(t)\,\ud\sigma(t).
\]
The Riemann-Lebesgue Lemma
implies that $(\widehat{f}(k))$
belongs to $c_\circ(\Ze)$,
the space of complex sequences
which vanish at infinity.
Since $\Te$ is compact,
$q\geq p$ implies
$L^q(\Te)\subset L^p(\Te)$.
In particular,
$L^2(\Te)\subset L^1(\Te)$.
The space $L^2(\Te)$ is
endowed with the inner product
\[
\scal{f}{g}=\int_\Te
f(t)\overline{g(t)}\,\ud\sigma(t)
\]
which turns $L^2(\Te)$ into a
Hilbert space.
The norm corresponding to
$\scal{\cdot}{\cdot}$ is
denoted by $\norm{\cdot}$.
The family $(e_k)_{k\in\Ze}$,
where $e_k(t)\eqd e^{ikt}$,
is an orthonormal basis of $L^2(\Te)$.
Bessel's inequality reads
\[
\norm{f}^2\geq
\sum_{k=-n}^n\mod{\widehat{f}(k)}^2,
\quad n\in\Ne.
\]
If $f\in L^2(\Te)$,
then the sequence $(\widehat{f}(k))$ of
its Fourier coefficients belongs to
\[
l_\Ce^2(\Ze)\eqd
\setc{u=(u_k)\in\Ce^\Ze}
{\sum_{k\in\Ze}\mod{u_k}^2<\infty},
\]
and Parseval's identity reads
\[
\norm{f}^2
=
\sum_{k\in\mathbb Z}
\mod{\widehat f(k)}^2.
\]
The space $l_\Ce^2(\Ze)$ is
endowed with the inner product
$$
\scal{u}{v}\eqd
\sum_{k\in\Ze} u_k \overline{v_k},
$$
which turns $l_\Ce^2(\Ze)$
into a Hilbert space.
We denote by $F$ the linear
mapping from $L^2(\Te)$ to
$l_\Ce^2(\Ze)$ which maps square
integrable functions
to their sequences of Fourier
coefficients.

We now move on to describe
convolution on the unit circle. Recall that, if
$\gamma\in L^1(\Te)$
and $f\in L^p(\Te)$ with
$p\in\segcc{1}{\infty}$, then the
function 
$s\mapsto\gamma(t-s)f(s)$ is
integrable for almost all
values of~$t$ and that
the function defined almost
everywhere on~$\Te$ by
$$
\big(\gamma\circledast f\big)(t)
\eqd
\int_\Te \gamma(t-s)f(s)\,\ud\sigma(s)
$$
belongs to $L^p(\Te)$ and
satisfies the inequality
$$
\normc{\gamma\circledast f}_p
\leq
\normc{\gamma}_1\normc{f}_p.
$$
In this paper, we are mostly
interested in the case
where $p=2$. The linear mapping
$T\colon f\mapsto Tf\eqd\gamma\circledast f$
then maps $L^2(\Te)$ into 
$L^2(\Te)$.

The convolution operator is diagonalized
by the Fourier transform, as summarized
by the commutative diagram
\begin{figure}[h]
\begin{picture}(150,100)(6,20)
\put(110,100){$L^2(\Te)$}
\put(205,100){$L^2(\Te)$}
\put(112,35){$l_\Ce^2(\Ze)$}
\put(207,35){$l_\Ce^2(\Ze)$}
\put(145,100){\vector(1,0){53}}
\put(145,35){\vector(1,0){53}}
\put(125,90){\vector(0,-1){40}}
\put(220,90){\vector(0,-1){40}}
\put(167,105){$T$}
\put(167,20){$\Gamma$}
\put(110,67){$F$}
\put(226,67){$F$}
\end{picture}
\end{figure}

Here,
$\Gamma\colon l_\Ce^2(\Ze)\to l_\Ce^2(\Ze)$ 
is the operator of elementwise
multiplication by 
$(\widehat\gamma(k))$:
if $s\in l_\Ce^2(\Ze)$, 
$$
\Gamma s\eqd \big(\widehat\gamma(k)\cdot s_k\big)_{k\in\Ze}.
$$
Recall indeed that
the $k$-th Fourier coefficient
of $\gamma\circledast f$ is
equal to
$\widehat\gamma(k)\cdot \widehat f(k)$
for every $k\in\Ze$.
It is easy to see that the
adjoint of~$\Gamma$ is the
operator of elementwise
multiplication by
$\big(\overline{\widehat\gamma(k)}\big)$.
The following theorem follows immediately from the $L^2$-convergence
of Fourier series.

\begin{Theorem}
The adjoint
$F^*\colon l_\Ce^2(\Ze)\to L^2(\Te)$ of~$F$ is given by
$$
F^\ast s(t)=\sum_{k\in\Ze}s_k e^{ikt}
$$
and satisfies $F^\ast=F^{-1}$,
so that~$F$ is unitary.
\end{Theorem}

As can be easily checked, for every $k\in\mathbb Z$,
\[
Te_k=\widehat\gamma(k)e_k.
\]
Thus, the functions \(e_k\), \(k\in\mathbb Z\), form an eigenbasis of \(T\),
with corresponding eigenvalues \(\widehat\gamma(k)\).
Consequently, \(T\) is injective
if and only if
\[
\widehat\gamma(k)\neq0,
\quad k\in\mathbb Z.
\]
More generally, the
kernel of~$T$ is given by
$$
\ker{T}=
\setc{f\in L^2(\Te)}
{\forall k\in\supp{\widehat{\gamma}},
\, \widehat{f}(k)=0}.
$$
where $\supp\widehat{\gamma}$
denotes the support of the
sequence $\widehat{\gamma}$, that is,
the set of indices~$k$ such
that $\widehat{\gamma}(k)\neq 0$.

\section{Deconvolution on the unit circle}
\label{sec:duc}

The ill-posedness of the deconvolution
problem can be characterized in a simple
way through the {\sl reduced minimum
modulus} of the operator~$T$. We define
\[
m(T):=\inf
\setc{\norm{Tf}}
{f\in(\ker{T})^\perp,\;\norm{f}=1}.
\]
Since~$T$ is diagonal in the Fourier
basis, 
$FTf(k)=\widehat\gamma(k)\,\widehat f(k)$,
so that, by Parseval's identity,
\[
m(T)=\inf
\setc{\mod{\widehat\gamma(k)}}
{\widehat\gamma(k)\neq 0}.
\]
Hence, if $\widehat\gamma$ has infinite support and
$|\widehat\gamma(k)|$ tends to zero along its support as
$|k|\to\infty$, then
\[
m(T)=0.
\]
By the Riemann-Lebesgue lemma, this is in particular the case
whenever~$\gamma$ belongs to $L^1(\Te)$ and $\widehat\gamma$ has infinite support.
Equivalently, the range of~$T$ is
not closed and the inverse, 
defined on $\ran{T}$, is not continuous.
This shows that the periodic deconvolution
problem is ill-posed in the sense of Hadamard.
This observation naturally motivates
the variational mollification approach
considered in this work.

\subsection{Mollification}

In this section, we focus on regularization
using approximate identities.
We show that the stable reconstruction
of a prescribed mollified target object is
possible.

An {\sl approximate identity}
on~$\Te$ is a family of functions
$(\varphi_\beta)_{\beta\in\segoc{0}{1}}$
in $L^1(\Te)$ that satisfies
\begin{enumerate}
\item[(a)]
for every $\beta\in\segoc{0}{1}$,
$\varphi_\beta\in L^1(\Te)$ and
$\int_\Te \varphi_\beta(t)\,\ud\sigma(t)=1$;
\item[(b)]
$(\varphi_\beta)$ is uniformly
bounded in the $L^1$-norm
by some positive constant~$C_1$;
\item[(c)]
for every $\delta\in(0,\pi)$,
\(
\int_{\mathbb T}
|\varphi_\beta(t)|\,
\1e_{(-\delta,\delta)^c}(t)\,\ud\sigma(t)
\rightarrow0
\)
as \(\beta\to0\), in which
$\1e_S(t)$ denotes, as usual,
the indicator function of the
set~$S$, meaning that
$\1e_S(t)$ is equal to~1 if
$t\in S$ and to~0 otherwise.
\end{enumerate}

The parameter~$\beta$ is
sometimes replaced by an
index~$n$ that goes to infinity,
yielding approximate unities
in the form of sequences.
The Fejér kernels are well-known examples of such approximate unities.

It is well-known that, if
$(\varphi_\beta)$ is an approximate
unity and $f\in L^p(\Te)$ with
$p\in\segco{1}{\infty}$, then
$\varphi_\beta\circledast f$ approaches~$f$
in $L^p$-norm as $\beta\downarrow 0$.
It is therefore natural to aim for
the reconstruction
of $\varphi_\beta\circledast f$
rather than~$f$. We shall see that
we obtain a regularization method
for the deconvolution problem.

Let $C_\beta$ denote the operator
of convolution by $\varphi_\beta$, which represents the prescribed reconstruction target at resolution~$\beta$.
The Morozov Completion
Condition between~$T$ and
$Q_\beta:=I-C_\beta$,
where $I$ denotes the identity of
$L^2(\Te)$,
asserts the existence
of a positive constant~$c_\beta$
such that
\begin{equation}
\label{eq:morozov}  
\forall f\in L^2(\Te),\quad
\normc{Tf}^2+
\normc{(I-C_\beta)f}^2
\geq
c_\beta\normc{f}^2.
\end{equation}
If the Morozov Completion Condition
between~$T$ and
$Q_\beta:=I-C_\beta$
is satisfied, then
the operator
$\big(T^\ast T+(I-C_\beta)^\ast
(I-C_\beta)\big)$ is boundedly
invertible, and the functional
$$
f\mapsto\frac{1}{2}
\normc{C_\beta g-Tf}^2+
\frac{1}{2}
\normc{(I-C_\beta)f}^2
$$
has a unique minimizer
$$
f_\beta=
\big(
T^\ast T+(I-C_\beta)^\ast
(I-C_\beta)\big)^{-1}T^\ast
C_\beta g
$$
which depends continuously on~$g$.
Expressed in terms of Fourier
coefficients, the solution
reads:
\begin{equation}
\label{eq:solution-fourier}
\widehat{f}_\beta(k)=
\frac{\overline{\widehat{\gamma}(k)}
\,\widehat\varphi_\beta(k)}
{\mod{\widehat{\gamma}(k)}^2+
\mod{1-\widehat\varphi_\beta(k)}^2}
\widehat{g}(k).
\end{equation}

\begin{Proposition}
\label{pro-morozov}
Assume that $\widehat\gamma(0)\neq0$.
If $\varphi_\beta$ is nonnegative and even, then the Morozov
Completion Condition between $T$ and $Q_\beta:=I-C_\beta$ is satisfied.
\end{Proposition}

\begin{proof}
Suppose that $\varphi_\beta$ is nonnegative
and even. Then
$\widehat \varphi_\beta(k)\in\Re$ and
\[
\widehat\varphi_\beta(k)\le
\widehat\varphi_\beta(0)=1,
\quad k\in\Ze.
\]
By unitarity of the Fourier transform,
for every $f\in L^2(\Te)$,
\[
\norm{Tf}^2+
\norm{(I-C_\beta)f}^2=
\sum_{k\in\mathbb Z}
\left(\mod{\widehat\gamma(k)}^2+
\mod{1-\widehat{\varphi_\beta}(k)}^2\right)
\mod{\widehat f(k)}^2.
\]
Hence
\[
\norm{Tf}^2+
\norm{(I-C_\beta)f}^2
\ge
\left(
\inf_{k\in\mathbb Z}
\bigl(\mod{\widehat\gamma(k)}^2+
\mod{1-\widehat{\varphi_\beta}(k)}^2\bigr)
\right)
\norm{f}^2.
\]
It remains to show that this infimum
is strictly positive.
Since $\varphi_\beta\in L^1(\mathbb T)$,
the Riemann-Lebesgue lemma implies
\[
\widehat{\varphi_\beta}(k)\to 0
\quad\text{as}\quad
\mod{k}\to\infty,
\]
hence
\[
\mod{1-\widehat{\varphi_\beta}(k)}\to 1.
\]
Therefore there exists $K\in\Ne$ such that
\[
\mod{1-\widehat{\varphi_\beta}(k)}
\ge\frac12
\quad\text{for all}\quad
\mod{k}\ge K.
\]
On the finite set $\{\mod{k}<K\}$,
the quantity
\[
\mod{\widehat\gamma(k)}^2+
\mod{1-\widehat{\varphi_\beta}(k)}^2
\]
is strictly positive.
Indeed, for $k=0$ we have
$\mod{\widehat\gamma(0)}>0$.
If $k\neq 0$, since $\varphi_\beta$ is nonnegative,
even, and normalized by
\[
\int_{\Te}\varphi_\beta\,\ud\sigma=1,
\]
we have
\[
\widehat{\varphi_\beta}(k)
=
\int_{\Te}\varphi_\beta(t)\cos(kt)\,\ud\sigma(t)
<1,
\]
since equality would require $\varphi_\beta$ to be supported,
up to a null set, on the finite set
$\set{t\in\Te}{\cos(kt)=1}$.
Hence
\[
\mod{1-\widehat{\varphi_\beta}(k)}>0.
\]
It follows that
\[
\inf_{k\in\Ze}
\left(
\mod{\widehat\gamma(k)}^2+
\mod{1-\widehat{\varphi_\beta}(k)}^2
\right)
>0,
\]
which proves the result.
\end{proof}

\subsection{Consistency}

The first theoretical result establishes the consistency of the proposed reconstruction.

\begin{Theorem}[Consistency of the mollified reconstruction]
Assume that
\[
g=Tf^\dagger,
\qquad
f^\dagger\in(\ker T)^\perp,
\]
and that, for every $k\in\mathbb Z$,
\[
\widehat{\varphi_\beta}(k)\to1
\qquad\text{as }\beta\to0.
\]
Assume moreover that the family $(\varphi_\beta)_{\beta\in(0,1]}$
is uniformly bounded in $L^1(\mathbb T)$, namely
\[
\sup_{\beta\in(0,1]}
\|\varphi_\beta\|_{L^1(\mathbb T)}<\infty.
\]
Then
\[
f_\beta\to f^\dagger
\qquad\text{in }L^2(\mathbb T)
\]
as $\beta\to0$.
\end{Theorem}

\begin{proof}
In the noise-free case, we have
\[
\widehat g(k)=\widehat\gamma(k)\widehat f^\dagger(k),
\]
and therefore
\[
\widehat f_\beta(k)
=
m_\beta(k)\widehat f^\dagger(k),
\]
where
\[
m_\beta(k)
=
\frac{|\widehat\gamma(k)|^2\,\widehat{\phi_\beta}(k)}
{|\widehat\gamma(k)|^2+|1-\widehat{\phi_\beta}(k)|^2}.
\]

Hence
\[
\widehat f_\beta(k)-\widehat f^\dagger(k)
=
\bigl(m_\beta(k)-1\bigr)\widehat f^\dagger(k),
\]
and by Parseval's identity,
\[
\|f_\beta-f^\dagger\|_{L^2(\mathbb T)}^2
=
\sum_{k\in\mathbb Z}
|m_\beta(k)-1|^2\,|\widehat f^\dagger(k)|^2.
\]

We now prove that the right-hand side tends to zero by dominated convergence.

First, fix $k\in\mathbb Z$. Since
\[
\widehat{\phi_\beta}(k)\to1
\qquad\text{as } \beta\to0,
\]
we have
\[
|1-\widehat{\phi_\beta}(k)|^2\to0.
\]
Thus
\[
m_\beta(k)
=
\frac{|\widehat\gamma(k)|^2\,\widehat{\phi_\beta}(k)}
{|\widehat\gamma(k)|^2+|1-\widehat{\phi_\beta}(k)|^2}
\to
1
\]
whenever $\widehat\gamma(k)\neq0$.
Since $f^\dagger\in(\ker T)^\perp$, we have
$\widehat f^\dagger(k)=0$ whenever $\widehat\gamma(k)=0$.
Hence only the indices for which $\widehat\gamma(k)\neq0$
contribute to the error.

It remains to find a summable dominating sequence. By the uniform
$L^1$-boundedness of the mollifiers,
\[
|\widehat{\phi_\beta}(k)|
\le
\|\phi_\beta\|_{L^1(\mathbb T)}
\le C
\]
for some constant $C>0$ independent of $k$ and $\beta$.

Consequently,
\[
|m_\beta(k)|
=
\left|
\frac{|\widehat\gamma(k)|^2\,\widehat{\phi_\beta}(k)}
{|\widehat\gamma(k)|^2+|1-\widehat{\phi_\beta}(k)|^2}
\right|
\le
|\widehat{\phi_\beta}(k)|
\le C.
\]
Hence
\[
|m_\beta(k)-1|
\le
|m_\beta(k)|+1
\le C+1.
\]
Therefore,
\[
|m_\beta(k)-1|^2|\widehat f^\dagger(k)|^2
\le
(C+1)^2|\widehat f^\dagger(k)|^2.
\]
Since $f^\dagger\in L^2(\mathbb T)$, we have
\[
\sum_{k\in\mathbb Z}|\widehat f^\dagger(k)|^2<\infty.
\]
Thus the dominated convergence theorem applies to the series, and we obtain
\[
\sum_{k\in\mathbb Z}
|m_\beta(k)-1|^2|\widehat f^\dagger(k)|^2
\to0.
\]
Hence
\[
\|f_\beta-f^\dagger\|_{L^2(\mathbb T)}\to0,
\]
which proves the result.
\end{proof}


\subsection{Convergence analysis}
\label{sec:con}

The proposed variational mollification method is built around
a prescribed mollified version of the exact solution, referred to as the
target object.
We now investigate the convergence of the corresponding reconstructions
towards the exact solution as the target resolution increases
($\beta\downarrow0$).

The analysis follows the
classical bias-noise decomposition used in regularization theory.
Its main ingredients are:
\begin{itemize}
\item
a bias-noise decomposition of the reconstruction error;
\item
a Besov-Nikolskii smoothness assumption on the exact solution;
\item
separate analyses of the ordinary smooth and supersmooth regimes.
\end{itemize}

The regularization parameter $\beta$ is finally obtained by balancing
the approximation and noise contributions.

\subsubsection{Bias-noise decomposition}

The reconstruction error admits the classical decomposition
into an approximation error and a noise propagation term.
Assume that the exact data satisfy
\[
g=Tf^\dagger,
\]
and that the noisy measurements obey
\[
\|g^\delta-g\|\le\delta.
\]
For $\beta>0$, let $f_\beta^\delta$ denote the regularized reconstruction defined by
\[
\widehat f_\beta^\delta(k)
=
\Phi_\beta(k)\widehat g^\delta(k).
\]
The triangle inequality yields the decomposition
\[
\|f_\beta^\delta-f^\dagger\|
\le
\|f_\beta-f^\dagger\|
+
\|f_\beta^\delta-f_\beta\|,
\]
where the first term is the approximation error (bias), while the second measures the propagation of the data perturbation.
In the present convolution setting, the reconstruction operator is
diagonalized by the Fourier basis. 
Introducing the Fourier multiplier
\[
m_\beta(k):=
\Phi_\beta(k)\widehat\gamma(k),
\]
Parseval's identity gives
\[
\|f_\beta-f^\dagger\|^2
=
\sum_{k\in\mathbb Z}
|1-m_\beta(k)|^2
|\widehat f^\dagger(k)|^2,
\]
whereas
\[
\|f_\beta^\delta-f_\beta\|
\le
\delta\,\|\Phi_\beta\|_{\ell^\infty(\mathbb Z)}.
\]
Consequently,
\[
\|f_\beta^\delta-f^\dagger\|
\le
\|f_\beta-f^\dagger\|
+
\delta\,\|\Phi_\beta\|_{\ell^\infty(\mathbb Z)}.
\]

The convergence analysis therefore reduces to estimating the bias term together with the norm
\[
\|\Phi_\beta\|_{\ell^\infty(\mathbb Z)},
\]
whose behavior depends on the decay of the Fourier coefficients of the kernel.

\subsubsection{Besov-Nikolskii smoothness}

To estimate the bias term, we assume that the exact solution belongs to a periodic Besov-Nikolskii space.
For $f\in L^2(\mathbb T)$ and $N\in\mathbb N$, define the Fourier tail
\[
E_f(N):=\sum_{|k|>N} |\widehat f(k)|^2.
\]
The periodic Besov-Nikolskii space $B^u_{2,\infty}(\mathbb T)$
is the set of functions $f\in L^2(\mathbb T)$ such that
\[
\sup_{N\geq 0}(1+N)^{2u}E_f(N)<\infty.
\]
We equip $B^u_{2,\infty}(\mathbb T)$ with the norm
\[
\|f\|_{B^u_{2,\infty}(\mathbb T)}^2
:=
|\widehat f(0)|^2
+
\sup_{N\geq 0}(1+N)^{2u}E_f(N).
\]
Throughout the sequel, we assume that
\[
f^\dagger\in B^u_{2,\infty}(\mathbb T).
\]

\subsubsection{Assumptions on the kernel}

Throughout the paper, we distinguish two classes of deconvolution problems according to the decay of the Fourier coefficients of the convolution kernel.

\paragraph{Ordinary smooth assumption.}
We say that the kernel is \emph{ordinary smooth} if there exist constants
$c_\gamma,C_\gamma>0$ and $b>0$ such that
\begin{equation}
\label{eq:ordinary-smooth}
c_\gamma(1+|k|)^{-b}
\le
|\widehat\gamma(k)|
\le
C_\gamma(1+|k|)^{-b},
\qquad k\in\mathbb Z.
\end{equation}
The parameter $b$ characterizes the degree of ill-posedness.

\paragraph{Supersmooth assumption.}
We say that the kernel is \emph{supersmooth} if there exist constants
$c_\gamma,C_\gamma>0$ and $a>0$ such that
\begin{equation}
\label{eq:supersmooth}
c_\gamma e^{-C_\gamma|k|^a}
\le
|\widehat\gamma(k)|
\le
C_\gamma e^{-c_\gamma|k|^a},
\qquad k\in\mathbb Z.
\end{equation}
In this case, the exponential decay of the Fourier coefficients leads to a severely ill-posed inverse problem.

The subsequent analysis is carried out separately under each of the two
kernel assumptions.

\paragraph{Notation.}
For two nonnegative quantities $A$ and $B$, we write
\[
A\lesssim B
\]
if there exists a constant $C>0$, independent of the regularization parameter $\beta$ and the noise level $\delta$, such that $A\le CB$. The notations
\[
A\gtrsim B
\quad\text{and}\quad
A\asymp B
\]
have the usual analogous meanings.

\subsubsection{Ordinary smooth case}

We now apply the general convergence strategy under the ordinary smooth
assumption
\eqref{eq:ordinary-smooth}.
Throughout this subsection we also assume
that the mollifier has approximation order $d>0$, namely
\begin{equation}
|1-\widehat{\phi_\beta}(k)|
\le
C_\phi(\beta |k|)^d,
\qquad k\in\mathbb Z,\ \beta\in(0,1].
\label{eq:mollifier-order}
\end{equation}
The analysis proceeds in three steps. We first estimate the bias,
then derive a stability estimate for the reconstruction operator,
and finally combine both estimates to obtain the convergence rate.

\begin{Lemma}[Bias estimate]
\label{lem:bias-ordinary}
Assume that
$f^\dagger\in B^{u}_{2,\infty}(\mathbb T)$
with
$0<u<b+d$.
Then
\[
\|f_\beta-f^\dagger\|_{L^2(\mathbb T)}
\lesssim
\beta^{\frac{du}{b+d}}.
\]
\end{Lemma}

\begin{proof}
We first estimate the bias term
\[
\|f_\beta-f^\dagger\|_{L^2(\mathbb T)}^2
=
\sum_{k\in\mathbb Z}
|1-m_\beta(k)|^2 |\widehat f^\dagger(k)|^2,
\]
where
\[
m_\beta(k)
=
\frac{|\widehat\gamma(k)|^2\widehat{\phi_\beta}(k)}
{|\widehat\gamma(k)|^2+|1-\widehat{\phi_\beta}(k)|^2}.
\]

Set
\[
A_k:=|\widehat\gamma(k)|^2,
\qquad
B_{\beta,k}:=|1-\widehat{\phi_\beta}(k)|.
\]
Then
\[
1-m_\beta(k)
=
\frac{A_k(1-\widehat{\phi_\beta}(k))+B_{\beta,k}^2}
{A_k+B_{\beta,k}^2},
\]
and hence
\begin{equation}
|1-m_\beta(k)|
\le
\frac{A_kB_{\beta,k}+B_{\beta,k}^2}
{A_k+B_{\beta,k}^2}.
\label{eq:mbeta-basic}
\end{equation}
The key ingredient is the transition scale obtained by balancing
the approximation error of the mollifier with the decay of the kernel:
\[
(\beta |k|)^d\sim (1+|k|)^{-b}.
\]
Thus we set
\begin{equation}
K_\beta
:=
\left\lfloor C_K\,\beta^{-\frac{d}{b+d}}\right\rfloor ,
\label{eq:Kbeta-ordinary}
\end{equation}
with $C_K>0$ chosen sufficiently small. Then, for $1\le |k|\le K_\beta$,
\[
B_{\beta,k}
\le
\frac12|\widehat\gamma(k)|.
\]
Using \eqref{eq:mbeta-basic}, we obtain on this low-frequency range
\[
|1-m_\beta(k)|
\lesssim
\frac{B_{\beta,k}}{|\widehat\gamma(k)|}.
\]
Invoking the ordinary smooth assumption
\eqref{eq:ordinary-smooth} together with the approximation property
\eqref{eq:mollifier-order}, we obtain
\begin{equation}
|1-m_\beta(k)|
\lesssim
(\beta |k|)^d(1+|k|)^b,
\qquad
1\le |k|\le K_\beta .
\label{eq:mbeta-ordinary}
\end{equation}

The transition scale naturally induces a decomposition into low-
and high-frequency contributions.
For the low-frequency part,
\eqref{eq:mbeta-ordinary} yields
\[
\sum_{1\le |k|\le K_\beta}
|1-m_\beta(k)|^2|\widehat f^\dagger(k)|^2
\lesssim
\beta^{2d}
\sum_{1\le |k|\le K_\beta}
|k|^{2(d+b)}|\widehat f^\dagger(k)|^2.
\]
By a dyadic decomposition and the Besov-Nikolskii tail estimate,
\[
\sum_{1\le |k|\le K_\beta}
|k|^{2(d+b)}|\widehat f^\dagger(k)|^2
\lesssim
K_\beta^{2(d+b-u)},
\qquad 0<u<d+b.
\]
Therefore,
\[
\sum_{1\le |k|\le K_\beta}
|1-m_\beta(k)|^2|\widehat f^\dagger(k)|^2
\lesssim
\beta^{2d}K_\beta^{2(d+b-u)}
\asymp
\beta^{\frac{2du}{b+d}}.
\]

For the high-frequency part, the uniform $L^1$ boundedness of the mollifiers
implies $|\widehat{\phi_\beta}(k)|\le C$, hence $|m_\beta(k)|\le C$ and
\[
|1-m_\beta(k)|\le C.
\]
Thus, using again the Besov-Nikolskii tail estimate,
\[
\sum_{|k|>K_\beta}
|1-m_\beta(k)|^2|\widehat f^\dagger(k)|^2
\lesssim
\sum_{|k|>K_\beta}|\widehat f^\dagger(k)|^2
\lesssim
K_\beta^{-2u}
\asymp
\beta^{\frac{2du}{b+d}}.
\]

Combining the low- and high-frequency estimates gives
\begin{equation}
\|f_\beta-f^\dagger\|_{L^2(\mathbb T)}
\lesssim
\beta^{\frac{du}{b+d}},
\qquad 0<u<d+b.
\label{eq:bias-ordinary}
\end{equation}
This concludes the proof.
\end{proof}

For the stability analysis, we further assume that the
mollifier family is uniformly bounded in $L^1(\mathbb T)$,
\begin{equation}
\sup_{\beta\in(0,1]}
\|\phi_\beta\|_{L^1(\mathbb T)}<\infty,
\label{eq:mollifier-L1-bound}
\end{equation}
and satisfies the lower damping estimates
\begin{equation}
|1-\widehat{\phi_\beta}(k)|
\gtrsim
(\beta|k|)^d,
\qquad
1\le |k|\le c/\beta
\label{eq:mollifier-lower-damping}
\end{equation}
and
\begin{equation}
|1-\widehat{\phi_\beta}(k)|
\ge c_0,
\qquad
|k|>c/\beta,
\label{HF}
\end{equation}
for $\beta$ sufficiently small.

\begin{Lemma}[Stability estimate]
\label{lem:stability-ordinary}
Under the above assumptions,
\[
\|\Phi_\beta\|_{\ell^\infty(\mathbb Z)}
\lesssim
\beta^{-\frac{db}{b+d}}.
\]
\end{Lemma}

\begin{proof}
By \eqref{eq:mollifier-L1-bound},
\[
|\widehat{\phi_\beta}(k)|\le C,
\]
and therefore
\begin{equation}
|\Phi_\beta(k)|
\lesssim
\frac{|\widehat\gamma(k)|}
{|\widehat\gamma(k)|^2+|1-\widehat{\phi_\beta}(k)|^2}.
\label{eq:Phi-basic-ordinary}
\end{equation}

As in the bias estimate, we split the analysis into the three frequency regions
\[
|k|\le K_\beta,\qquad
K_\beta<|k|\le c/\beta,
\qquad
|k|>c/\beta,
\]
where
\[
K_\beta
\asymp
\beta^{-\frac{d}{b+d}}.
\]

In the low-frequency region,
\eqref{eq:ordinary-smooth} immediately yields
\[
|\Phi_\beta(k)|
\lesssim
(1+|k|)^b
\lesssim
\beta^{-\frac{db}{b+d}}.
\]

In the intermediate-frequency region,
\eqref{eq:mollifier-lower-damping} yields
\[
|\Phi_\beta(k)|
\lesssim
\frac{(1+|k|)^{-b}}
     {(\beta|k|)^{2d}}
\lesssim
\beta^{-\frac{db}{b+d}}.
\]

Finally, by \eqref{HF}, for $|k|>c/\beta$,
\[
|\Phi_\beta(k)|
\lesssim
|\widehat\gamma(k)|
\lesssim
(1+|k|)^{-b},
\]
which is uniformly bounded and therefore of smaller order than
$\beta^{-db/(b+d)}$.

Combining the three estimates gives
\[
\|\Phi_\beta\|_{\ell^\infty(\mathbb Z)}
\lesssim
\beta^{-\frac{db}{b+d}}.
\]
The proof is complete.
\end{proof}

\begin{Theorem}[Order-optimal convergence rate]
\label{thm:ordinary}
Assume that
$f^\dagger\in B^{u}_{2,\infty}(\mathbb T)$,
with
$0<u<b+d$,
and choose
\[
\beta(\delta)
\asymp
\delta^{\frac{b+d}{d(u+b)}}.
\]
Then
\[
\|f_{\beta(\delta)}^\delta-f^\dagger\|_{L^2(\mathbb T)}
\lesssim
\delta^{\frac{u}{u+b}}.
\]
\end{Theorem}

\begin{proof}
Combining Lemmas~\ref{lem:bias-ordinary}
and~\ref{lem:stability-ordinary}
with the deterministic error decomposition yields
\[
\|f_\beta^\delta-f^\dagger\|
\lesssim
\beta^{\frac{du}{b+d}}
+
\delta
\beta^{-\frac{db}{b+d}}.
\]
Balancing the approximation and stability terms gives
\[
\beta(\delta)
\asymp
\delta^{\frac{b+d}{d(u+b)}},
\]
which immediately yields the stated estimate.
\end{proof}

\begin{Remark}[Order-optimality and saturation]

The convergence rate established in Theorem~\ref{thm:ordinary},
\[
\|f_{\beta(\delta)}^\delta-f^\dagger\|_{L^2(\mathbb T)}
=
O\!\left(\delta^{\frac{u}{u+b}}\right),
\]
is order-optimal, in the minimax sense, for deterministic
deconvolution problems with polynomially decaying Fourier coefficients
\[
|\widehat\gamma(k)|\asymp (1+|k|)^{-b}
\]
and Besov-Nikolskii smoothness
$f^\dagger\in B^u_{2,\infty}(\mathbb T)$.
The qualification of the method is determined by the approximation order
$d$ of the mollifier. This is reflected by the restriction
\[
0<u<b+d,
\]
which appears in the proof of the bias estimate. In other words,
increasing the smoothness of the exact solution beyond this threshold
does not further improve the convergence rate unless the approximation
order of the mollifier is increased accordingly.
\end{Remark}

\subsubsection{Supersmooth case}

We now apply the general convergence strategy under the supersmooth
assumption \eqref{eq:supersmooth}.
Throughout this subsection, we retain the mollifier assumptions
\eqref{eq:mollifier-order} and \eqref{eq:mollifier-L1-bound}.
Thus,
\[
c_\gamma e^{-C_\gamma |k|^a}
\le
|\widehat\gamma(k)|
\le
C_\gamma e^{-c_\gamma |k|^a},
\qquad k\in\mathbb Z,
\]
for some constants $c_\gamma,C_\gamma>0$ and $a>0$.
The convergence analysis follows exactly the same three-step strategy as
in the ordinary smooth case. We first estimate the approximation error,
then derive a stability estimate for the reconstruction operator, and
finally combine both estimates to obtain the convergence rate.

\begin{Lemma}[Bias estimate]
\label{lem:bias-supersmooth}
Assume that
$f^\dagger\in B^u_{2,\infty}(\mathbb T)$.
Then
\[
\|f_\beta-f^\dagger\|_{L^2(\mathbb T)}
\lesssim
(\log(\beta^{-1}))^{-u/a}.
\]
\end{Lemma}

\begin{proof}
As in the ordinary smooth case,
\[
\|f_\beta-f^\dagger\|_{L^2(\mathbb T)}^2
=
\sum_{k\in\mathbb Z}
|1-m_\beta(k)|^2
|\widehat f^\dagger(k)|^2,
\]
where
\[
m_\beta(k)
=
\frac{|\widehat\gamma(k)|^2\widehat{\phi_\beta}(k)}
     {|\widehat\gamma(k)|^2+|1-\widehat{\phi_\beta}(k)|^2}.
\]
The same balancing principle as in the ordinary smooth case now leads
to a logarithmic transition scale. Balancing the approximation error of
the mollifier with the exponential decay of the kernel yields
\begin{equation}
K_\beta
:=
\left\lfloor
\eta(\log(\beta^{-1}))^{1/a}
\right\rfloor,
\label{eq:Kbeta-supersmooth}
\end{equation}
where $\eta>0$ is chosen sufficiently small.
Proceeding as in the proof of the ordinary smooth case, one obtains
\begin{equation}
|1-m_\beta(k)|
\lesssim
\beta^\rho |k|^d,
\qquad
1\le |k|\le K_\beta,
\qquad
\rho:=d-C_\gamma\eta^a>0.
\label{eq:mbeta-supersmooth}
\end{equation}
The transition scale again induces a decomposition into low- and
high-frequency contributions. Using the dyadic estimate established in
the ordinary smooth case, the low-frequency contribution satisfies
\[
\sum_{|k|\le K_\beta}
|1-m_\beta(k)|^2
|\widehat f^\dagger(k)|^2
\lesssim
\beta^{2\rho}
K_\beta^{2(d-u)_+}.
\]
Since $K_\beta$ grows only logarithmically whereas
$\beta^{2\rho}$ decays algebraically, this term is negligible.
For the high-frequency contribution, the uniform estimate
$|1-m_\beta(k)|\lesssim1$ together with the Besov-Nikolskii tail estimate
gives
\[
\sum_{|k|>K_\beta}
|1-m_\beta(k)|^2
|\widehat f^\dagger(k)|^2
\lesssim
K_\beta^{-2u}
\asymp
(\log(\beta^{-1}))^{-2u/a}.
\]
Consequently,
\begin{equation}
\|f_\beta-f^\dagger\|_{L^2(\mathbb T)}
\lesssim
(\log(\beta^{-1}))^{-u/a}.
\label{eq:bias-supersmooth}
\end{equation}
This concludes the proof.
\end{proof}

\begin{Lemma}[Stability estimate]
\label{lem:stability-supersmooth}
Assume that the mollifier family satisfies
\eqref{eq:mollifier-lower-damping}
and \eqref{HF}.
Then
\[
\|\Phi_\beta\|_{\ell^\infty(\mathbb Z)}
\lesssim
\beta^{-d}.
\]
\end{Lemma}

\begin{proof}
By the uniform $L^1$-boundedness of the mollifier family,
\[
|\widehat{\phi_\beta}(k)|\le C,
\]
and therefore
\begin{equation}
|\Phi_\beta(k)|
\lesssim
\frac{|\widehat\gamma(k)|}
     {|\widehat\gamma(k)|^2+|1-\widehat{\phi_\beta}(k)|^2}.
\label{eq:Phi-basic-supersmooth}
\end{equation}

Using the same logarithmic transition scale
\[
K_\beta
\asymp
(\log(\beta^{-1}))^{1/a},
\]
we distinguish three frequency regions.

For $|k|\le K_\beta$, the supersmooth assumption
\eqref{eq:supersmooth} gives
\[
|\Phi_\beta(k)|
\lesssim
e^{C_\gamma K_\beta^a}
\lesssim
\beta^{-d},
\]
provided that $\eta>0$ is chosen sufficiently small.
For $K_\beta<|k|\le c/\beta$, we use
\[
A^2+B^2\ge 2AB,
\qquad A,B\ge0,
\]
with
\[
A=|\widehat\gamma(k)|,
\qquad
B=|1-\widehat{\phi_\beta}(k)|.
\]
Together with \eqref{eq:Phi-basic-supersmooth}, this gives
\[
|\Phi_\beta(k)|
\lesssim
\frac{1}{|1-\widehat{\phi_\beta}(k)|}.
\]
Invoking \eqref{eq:mollifier-lower-damping}, we obtain
\[
|\Phi_\beta(k)|
\lesssim
(\beta|k|)^{-d}
\lesssim
\beta^{-d}.
\]

Finally, for $|k|>c/\beta$, condition \eqref{HF} yields
\[
|\Phi_\beta(k)|
\lesssim
|\widehat\gamma(k)|
\lesssim
1
\lesssim
\beta^{-d}.
\]

Consequently,
\begin{equation}
\|\Phi_\beta\|_{\ell^\infty(\mathbb Z)}
\lesssim
\beta^{-d}.
\label{eq:amp-supersmooth}
\end{equation}
The proof is complete.
\end{proof}

\begin{Theorem}[Order-optimal logarithmic convergence rate]
\label{thm:supersmooth}
Assume that
$f^\dagger\in B^u_{2,\infty}(\mathbb T)$
and choose
\[
\beta(\delta)
=
\delta^{1/(2d)}.
\]
Then
\[
\|f_{\beta(\delta)}^\delta-f^\dagger\|_{L^2(\mathbb T)}
\lesssim
(\log(\delta^{-1}))^{-u/a}.
\]
\end{Theorem}

\begin{proof}
Combining
Lemmas~\ref{lem:bias-supersmooth}
and~\ref{lem:stability-supersmooth}
with the deterministic error decomposition yields
\[
\|f_\beta^\delta-f^\dagger\|
\lesssim
(\log(\beta^{-1}))^{-u/a}
+
\delta\,\beta^{-d}.
\]
Choosing
\(
\beta(\delta)=\delta^{1/(2d)}
\)
gives a noise contribution of order
\(O(\delta^{1/2})\),
which is negligible compared with
\(
(\log(\delta^{-1}))^{-u/a}
\).
The result follows immediately.
\end{proof}

\begin{Remark}[Order-optimality in the supersmooth case]
Under the supersmooth assumption
\[
|\widehat\gamma(k)|\asymp e^{-c|k|^a},
\]
the convergence rate
\[
\|f_{\beta(\delta)}^\delta-f^\dagger\|_{L^2(\mathbb T)}
=
O\!\left((\log(\delta^{-1}))^{-u/a}\right)
\]
is order-optimal, in the minimax sense, for exponentially ill-posed
deconvolution problems with Besov-Nikolskii smoothness
$f^\dagger\in B^u_{2,\infty}(\mathbb T)$.
Unlike the ordinary smooth case, no saturation phenomenon occurs:
the logarithmic convergence rate is entirely dictated by the exponential
decay of the Fourier coefficients of the kernel and therefore cannot be
improved uniformly over bounded subsets of
$B^u_{2,\infty}(\mathbb T)$.
\end{Remark}

\section{Numerical Results: Reconstruction of Circular Wind Profiles}
\label{sec:num}

In this section, to validate theoretical results, we use our framework on a concrete inverse problem arising in environmental monitoring: the recovery of wind direction probability densities from sensor data.

\subsection{Problem Formulation and Data Simulation}
The direction of the wind is a circular quantity defined on the unit circle. Accurate estimation of directional modes is critical for wind energy harvesting, where the joint probability of wind speed and direction determines available energy potential \cite{carta2008}, and precise directional data is a prerequisite for optimizing turbine layout to minimize wake losses and maximize array efficiency \cite{herbert2014}.

We simulate a ground truth signal $f \in L^2(\mathbb{T})$ representing a bimodal wind profile, constructed as a mixture of two Wrapped Gaussian (WG) distributions:
\begin{equation}
    f(\theta) = 0.65 \cdot \mathcal{WG}(\theta; 4.7, 0.5) + 0.35 \cdot \mathcal{WG}(\theta; 0.8, 0.25),
\end{equation}
where the dominant mode corresponds to a Westerly wind and the secondary mode to a North-Easterly current. The observations are discretized on a grid of size $N=512$.

The sensor blurring is modeled by the periodic kernel $\widehat{\gamma}_k = e^{-\alpha k^2}$ with diffusion parameter $\alpha=0.01$. The convolved data is further corrupted by additive white Gaussian noise at a relative level of $\delta \approx 20\%$. As shown in Figure \ref{fig:results}(b), this level of degradation effectively merges the two distinct modes into a single broad distribution, rendering the secondary peak indistinguishable in the raw data.

\subsection{Parameter Selection via the L-curve Method}
A pervasive challenge in operational settings is the absence of a priori knowledge regarding the exact noise level $\delta$, which precludes the direct use of the Discrepancy Principle. To address this "blind" deconvolution scenario, we employ the L-curve criterion \cite{hansen1992}.

We compute the regularized solution $f_\beta$ for a logarithmic range of the mollification parameter $\beta \in [10^{-5}, 10^{-1}]$. We plot the trade-off curve parameterized by $\beta$ in the log-log scale:
\begin{equation}\label{eq:lcurve}
    \mathcal{L}(\beta) = \left( \log \| T f_\beta - g^\delta \|_{L^2}, \,\, \log \| f_\beta \|_{L^2} \right).
\end{equation}
The optimal parameter $\beta_{\text{opt}}$ is identified as the vertex (corner) of this curve, where the curvature
\begin{equation}\label{eq:curvature}
    \kappa(\beta) = \frac{u'v'' - v'u''}{(u'^2 + v'^2)^{3/2}}
\end{equation}
is maximal, where $u = \log\|Tf_\beta - g^\delta\|_{L^2}$, $v = \log\|f_\beta\|_{L^2}$, and primes denote derivatives with respect to $t = \log\beta$.
This point provides a practical balance between data fidelity
and regularization (see Figure \ref{fig:lcurve}).
\begin{figure}[H]
    \centering
    \includegraphics[width=\textwidth]{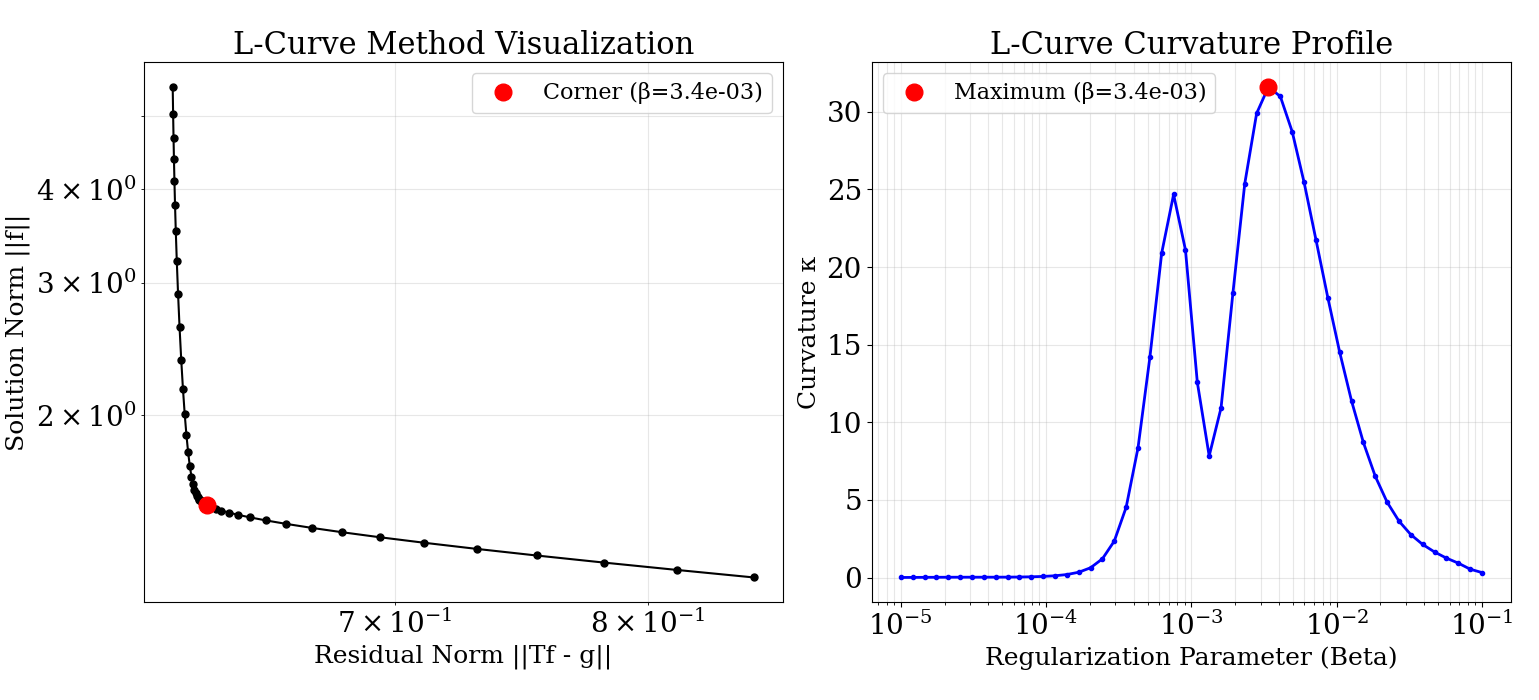}
    \caption{Parameter selection via the L-curve. Left: The L-curve in
    log-log scale; the optimal $\beta$ (red dot) is located at the corner where
    curvature is maximal. Right: The curvature profile $\kappa(\beta)$,
    confirming the identified corner corresponds to a sharp, well-defined maximum.}
    \label{fig:lcurve}
\end{figure}

\subsection{Stability Analysis}

To assess the L-curve selection, we examine the reconstruction error
and numerical stability in Figure~\ref{fig:error_analysis}.
The relative error (left panel) exhibits a well-defined minimum, and the
value of $\beta$ selected by the L-curve criterion (red dot) lies close
to this minimum. Thus, in this experiment, the L-curve provides a
near-optimal parameter choice without requiring knowledge of the ground
truth. The condition number shown in the right panel provides a
complementary view of the regularization effect: at the selected value
of $\beta$, the condition number is reduced from approximately
$10^{22}$ to $10^{20}$, while remaining large, as expected for this
severely ill-posed problem.

\begin{figure}[H]
    \centering
    \includegraphics[width=\textwidth]{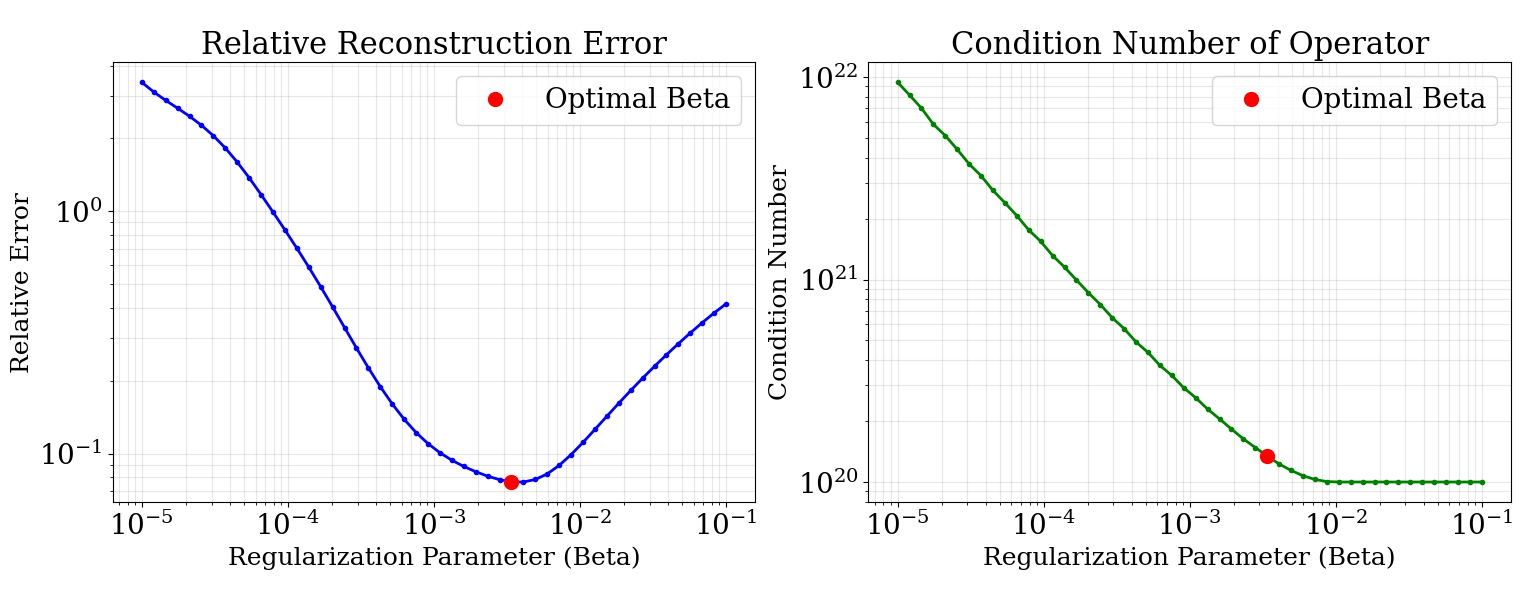}
    \caption{Stability metrics. Left: The relative error showing the optimal $\beta$ yields minimum error. Right: The condition number decreases with regularization, quantifying the gain in numerical stability.}
    \label{fig:error_analysis}
\end{figure}
\subsection{Reconstruction Results}
The numerical results are presented in Figure \ref{fig:results}. The L-curve method identifies an optimal parameter of $\beta \approx 3.4 \times 10^{-3}$. As a post-processing step, we enforce physical consistency by applying a non-negativity projection followed by renormalization:
\begin{equation}
\label{eq:projection}
    f_\beta \leftarrow \frac{\max(0,\, F^{-1}(\widehat{f}_\beta))}%
    {\displaystyle\sum_{i} \max\!\left(0,\, F^{-1}(\widehat{f}_\beta)_i\right)},
\end{equation}
so that the reconstruction integrates to unity and constitutes a valid discrete probability distribution, consistent with the discretization of $f$. We note that the theoretical convergence results of Section~\ref{sec:con} are established for the unconstrained $f_\beta$; the projection is applied solely to ensure interpretability of the final estimate.

Despite significant noise and blurring, the reconstruction successfully disentangles the bimodal structure of the wind profile. The secondary North-Easterly mode, which was masked in the observation $g^\delta$, is recovered with accurate localization. This demonstrates the stability of the mollified inversion scheme and the efficacy of the L-curve criterion in handling blind deconvolution problems on the circle.

\begin{figure}[h]
    \centering
    \includegraphics[width=\textwidth]{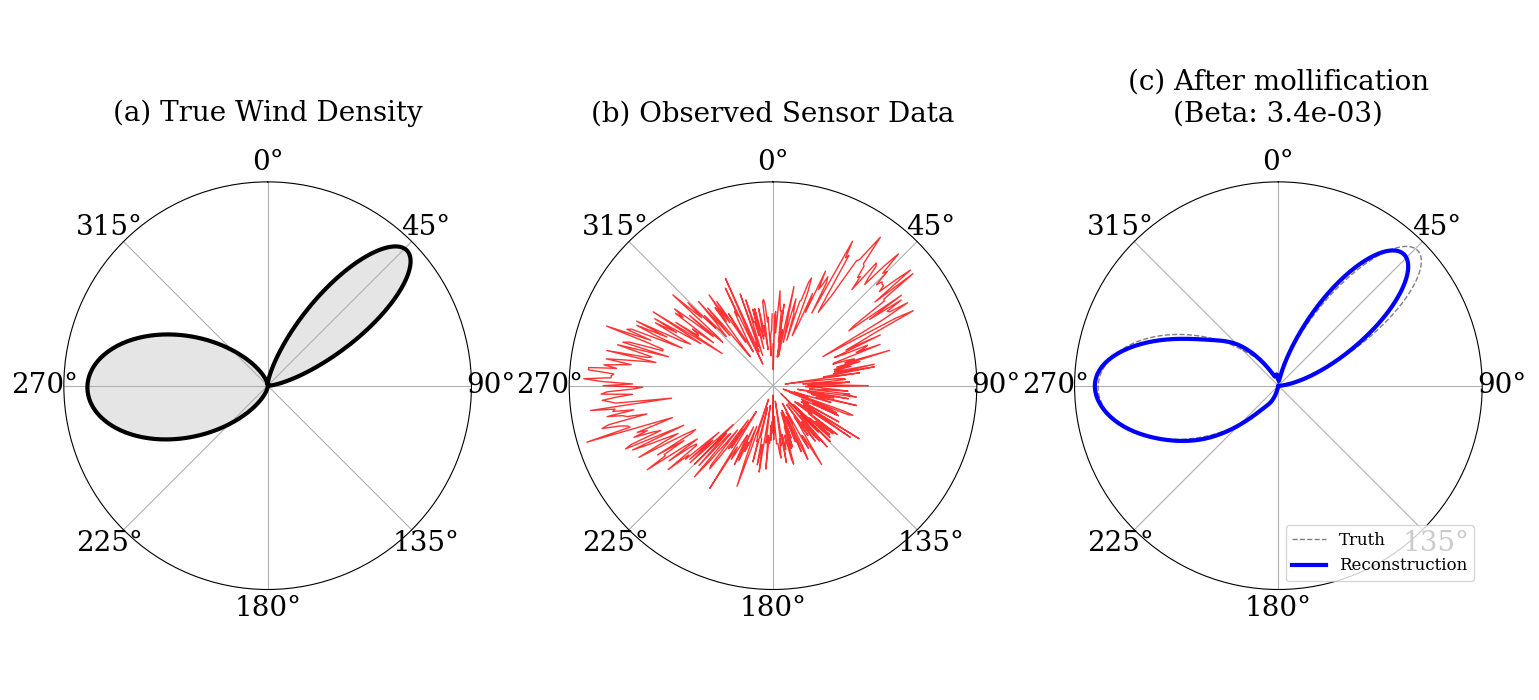}
    \caption{Reconstruction of circular wind density. (a) Ground truth bimodal distribution. (b) The observed sensor data ($\delta \approx 20\%$), where the secondary mode is obscured by noise and blur. (c) The reconstruction using the optimal $\beta$, recovering the bimodal structure.}
    \label{fig:results}
\end{figure}

\bibliographystyle{siam}
\bibliography{references}

\end{document}